 \documentclass[12pt]{amsart}

\usepackage{amsmath,amsthm}
\usepackage{enumitem}
\usepackage{theoremref}

\usepackage{tikz}
\usetikzlibrary{decorations.pathreplacing,patterns}

\newtheorem{theorem}{Theorem}

\newtheorem{lemma}[theorem]{Lemma}

\newtheorem*{lemmata}{Lemma}

\newtheorem{corollary}[theorem]{Corollary}
\newtheorem{proposition}[theorem]{Proposition}
\newtheorem{remark}[theorem]{Remark}

\newcounter{Enum}					
\newenvironment{Enumerate}{\begin{enumerate}[label={\rm(\roman*)}]}{\end{enumerate}}

	\title{Localization of trigonometric polynomials and the Lev-Tselishchev lemma}
\author{Bochkov Ivan}
\address{SPB SU}
\email{bv1997@ya.ru}
\begin{document}
	%
	\begin{abstract}
	The paper shows that the localization lemma used by N. Lev and A. Tselishchev in solving the problem of constructing quasi-bases from uniformly separated shifts of a function in the space $L^p (\mathbb R) $ for $ p > ( 1 + \sqrt 5 )/2 $, is not carried over to small $p$.
	\end{abstract} 
	
	\maketitle
	
In the series of recent works by N. Lev and A. Tselishchev \cite{LT, LT2}, the problem of the existence of a quasibasis (Schauder frame) in the space $L^p(\mathbb{R})$, $1 \le p \le 2$, consisting of shifts of one function, was studied. The main result of these works is that a unconditional quasibasis of this type does not exist for all $p$ in the considered interval, and for $p > (1 + \sqrt{5})/2$, there exists a quasibasis consisting of uniformly spaced shifts of a function.

A key role in the proof of the latter result is played by the following lemma. Let $Q = \sum_m q_m e^{imt}$, $t \in \mathbb{R}$, be a trigonometric polynomial. For $p \in [1, \infty)$, we set $|Q|_{A_p} = \left(\sum_m |q_m|^p\right)^{1/p}$. We also denote $S_N(Q) = \sum_{m=-N}^N q_m e^{imt}$ the partial Fourier sum for $ Q$.

\begin{lemmata}\label{LT} Let $p > (1 + \sqrt{5})/2$. Then for any $\epsilon > 0$, there exist trigonometric polynomials $P = \sum_m p_m e^{imt}$ and $Q = \sum_m q_m e^{imt}$, $t \in \mathbb{R}$, such that
	
	\begin{Enumerate}
		\item $q_0 = 0$, $\max_j |q_j| < \epsilon$;
		\item $|P - 1|_{A_p} < \epsilon$;
		\item $|PQ - 1|_{A_p} < \epsilon$;
		\item $|PS_m(Q)|_{A_p} < C_p$ for some constant $C_p$ depending only on $p$.
	\end{Enumerate}
\end{lemmata}

The condition $p > (1 + \sqrt{5})/2$ in the main result of Lev-Tselishchev arises from this lemma. The authors of \cite{LT} noted that the validity of the lemma for all $p > 1$ would imply the existence of the desired Schauder frame for all $p > 1$.

The purpose of this paper is to demonstrate that for $p < (1 + \sqrt{5})/2$, this lemma is false, which means that the question of the Schauder frame from shifts of a function cannot be solved by this method. Thus, there is a unique value for which the validity of the lemma is unknown, namely $p = (1 + \sqrt{5})/2$.

\section{Notations}

$ \mathbb T $ is the unit circle identified with the interval $ [0 , 2\pi ] $. We use $ | M | $ for the Lebesgue measure of a subset $ M \subset  \mathbb T $. 

The notation $x = \theta(y)$ frequently used in the article means that $y = O(x)$. The notation $x \succeq y$ means that $x \ge Cy$ for some absolute constant $C$. In this paper, the exponent $p$ will be fixed, i.e., all constants may depend on it. 

We will naturally identify measurable complex-valued functions on the unit circle with their trigonometric series. Then, $ f_m $ stands for the Fourier coefficients of a function $ f \in L^1 ( \mathbb T ) $. By $|F|_{L^p}$ we denote the $L^p$ norm of $F$ on the circle, and by $|F|_{A^p}$ the $l^p$ norm of its Fourier coefficients. The classical Hausdorff-Young inequality states that $|F|_{A^p} \ge c_p |F|_{L^q}$, $ p \in [1,2] $, where $c_p$ is an absolute constant depending on $p$, and $q = p/(p-1) $ is the dual exponent to $p$. Then, for any integrable function $ f$ on the unit circle we define $S_N(f) = \sum_{m=-N}^N f_m e^{imt}$, the partial Fourier sum. Occasionally we use $| \cdot|_{l^p}$ for the $l^p$-norm.

For two functions $A, B$ on the unit circle, $A* B$ stands for their convolution.

For a segment of the circle $I$ and a function $F$, we denote by $\mathbb{E}_I(F)$ the average of $F$ over the segment $I$ with respect to the Lebesgue measure unless otherwise specified. By $\mathbb{E}(F)$ we denote the average of $F$ over the entire circle. By $\int F$ we denote the integral of $F$ over the entire circle.

By $\widehat{A}(m)$ we denote the $m$-th Fourier coefficient of the function $A$.

For vectors $u, v$, by $(u, v)$ we denote their scalar product.

By $\chi_k$ we denote the characteristic function of the segment $[-k, k]$. By $\chi_M$ for an arbitrary set $M$ we denote the characteristic function of this set.

By $\mathbb{VAR}(F)$ we denote the variation of the function $F$.

By $D_m$ we denote the $m$-th Dirichlet kernel, $D_m(x) = \sum_{k=-n}^n \frac{e^{ikx}}{2}$.

By $\Re F$ we denote the real part of the function $F$.

A compactly supported function $f$ on the real  line will sometimes be identified with a function $F$ on the circle as follows, $F(x) = \sum_{n \in \mathbb{Z}} f(x + 2\pi n)$. For a function $f$ defined on $[-\pi, \pi]$, this coincides with the classical understanding.

\section{Main result}

\begin{theorem}\label{main} Let $ 1 < p < \frac{1 + \sqrt{5}}{2} $, $q$ be the dual exponent, $C > C_p > 1$ be a sufficiently large real number, $1 > \delta > 0$ be an arbitrary fixed real number. Let $ \widetilde A \in A_1 $ be such that for a function $ \tilde B \in A_1 $ the following assumptions are obeyed,
	
	1) $\mathbb{E}(\widetilde{B}) = 1$,

	2) $\Re(\mathbb{E}(\widetilde{A})) \leq C^{-1}$,
	
	3) $\max_m | (1 - \widetilde{A}) S_m(\widetilde{B}) |_{A_p} \preceq 1$
	
	4) $| (1 - \widetilde{A}) \widetilde{B} |_{A_p} < 1 - \delta$.
	
Then
	$$\label{norm} \left| \widetilde{A} \right|_{A_p} \succeq C^{\frac{q(2-p)-1}{p}-\epsilon}$$ for any $ \epsilon > 0 $, where the constants here and in point 3) may depend on $p$.
	\end{theorem}

Notice that $\frac{q(2-p)-1}{p} > 0$ whenever $p < \frac{1 + \sqrt{5}}{2}$.
	
	From this, it follows that the statement of Lemma \ref{LT} is not satisfied for any $p \in (1, \frac{1 + \sqrt{5}}{2})$. Indeed, if it were satisfied for some $p$ in this interval, then for any given $C > 1$, taking sufficiently small $\epsilon$ and setting $\widetilde{A} = 1 - P$, $\widetilde{B} = 1 - Q$, we would obtain trigonometric polynomials $\widetilde{A}$ and $\widetilde{B}$ satisfying the conditions of the theorem for a given $D$. Indeed, conditions (i), (iii), and (iv) are checked directly, and condition (ii) follows from the fact that as the $A_p$ norm tends to 0, then $|\widetilde{\widehat{A}}(0)| \to 0$ as well, and hence $\Re(\mathbb{E} A) \to 0$.
	
The result of the theorem now obviously contradicts condition (ii) of the lemma.
	
First, we prove a lemma that will allow us to estimate the $A_p$-norm of a function through an object called a $T$-sequence in the article.
	
	\section{$T$-sequences and key Lemma}
	
	We call a $T$--system $\mathcal{TS}_m = (F_m, T_m)$ a tuple of a natural number $m$, a function $F_m$ on the unit circle $ \mathbb T $, and a finite set of disjoint segments $T_m = \{T_{m,1}, T_{m,2}, \ldots T_{m,k_m}\}$, $ T_{ m, j } \subset \mathbb T $, with the following properties:
	
	\begin{enumerate}
		\item $F_m = 0$ outside $T_m$,
		\item $\forall i \quad \int_{T_{m,i}} F_m = 0$,
		\item $\forall x \quad |F_m(x)| \leq 1$,
		\item $\forall j \quad |T_{m,j}| \leq \frac{2\pi}{2^m}$
		\item $\mathbb{VAR}(F_m) \prec k_m$.
	\end{enumerate}
	
	For a $T$ system, we denote $\mu_m = \frac{k}{2^m}$.
	
	A $T$-sequence is a sequence of $T$-systems ${\mathcal{TS}_0, \mathcal{TS}_1, \ldots, \mathcal{TS}_k, \ldots}$.
	
	\begin{lemma}\label{T-sec} Fix some real numbers $p, r, \alpha$, such that $2 > p, r > 1$, $2 - p > \alpha > 0$. Let for some real $M, C$, such that $M \geq 1$, $C \geq 1$, there exist a $T$-sequence $\mathcal{P}$ and a complex-valued function $A$ on the circle with the following properties:
		\label{eq:lemma2}
		\begin{enumerate}
			\item $|\mathbb{E}(F_i A)| \succeq \frac{\mu_i}{M}$,
			\item $\mu_i \preceq C^{-r}$,
			\item $\sum_i \mu_i \succeq \frac{M^{2-\alpha}}{C}$.
		\end{enumerate}
		
		Then $|A|_{A_p} \succeq C^{\frac{r(2-p)-1}{p}}$ (the implied constant may depend on $p, r, \alpha$).
		
	\end{lemma}
\begin{proof}
	
Given an $ j \in \mathbb N $ denote by $P_j(F)$ for a function $F\in L^1 ( \mathbb T ) $ the projection onto the space of functions whose Fourier transform has support on the union of segments $[-2^{j+1}, -2^j]$ and $[2^j, 2^{j+1}]$, that is, 

\[ P_j ( F ) (x) = \sum_{2^j \le |n| < 2^{j+1} } F_n e^{inx} . \]

Also denote $A_j = |P_j(A)|_{A_p}$. To prove the lemma, we need to estimate $\sum A_i$.

\begin{proposition} For any $ j \in \mathbb N $ \begin{equation}\label{Pr3} \left\|  P_j(F_m)\right\|_{L_2(\mathbb T ) }^2 \preceq \frac{\mu_m}{2^{|j-m|}}. \end{equation} \label{eq:prop3}
	\begin{proof}
		
Notice first that $\left\| P_j(F_m)\right\|^2_{L_2 (\mathbb T)} \leq \left\| F_m \right\|^2_{L^2(\mathbb T) } \leq \mu_m $, so it suffices to consider $|m-j| > 10 $.
		
For $j < m$, consider $T = 2^{j+2} F_m *\chi_{2^{-j-2}}$. We have
$$\left\|T\right\|_{{L_2}}^2 = 2^{2j} \sum_{ t \in \mathbb Z } 
		\left|
		\widehat{F_m}(t)
		\right|^2 \widehat{\chi_{2^{-j-2}}}(t)^2 \succeq \sum_{2^j \leq |t| < 2^{j+1}} \left|\widehat{F_m}(t) \right|^2 = \left\|P_j(F_m)\right\|_{{L_2}}^2 $$ 
		since $\left|\widehat{\chi_{2^{-j-2}}}(t) \right| \succeq 2^{-j} $ for $2^j \leq |t| < 2^{j+1}$. 
According to assumption (2) from the definition of $T$-systems the means of $ F_m $ over the intervals of the system vanish, hence $ T (x ) $ vanishes whenever neither end of the interval $ ( x - 2^{-j-2}, x + 2^{-j-2} ) $ belongs to a segment from $ T_m $, and, if either of them does, the integration in the definition of $ T $ actually goes over that segment(s) only. These two observations imply, respectively, that, first, the meausure of the support of $T $ is $O(\mu_m)$, and second, that $|T(x)| \preceq 2^{j-m}$ for any $ x$. It follows that $\left\|  P_j(F_m)\right\|^2_{L_2(\mathbb T ) } \preceq \left\|T \right\|_{L_2}^2 \preceq \frac{\mu_m}{2^{2|j-m|}}$, as required.
		
		For $ j > m$, consider $T = F_m - 2^{j-2} F_m* \chi_{2^{2-j}}$. Then $ \left\|  P_j(F_m)\right\|^2_{L_2(\mathbb T ) } \preceq \left\|T\right\|_{L_2}^2$ similarly to the previous case. On the other hand, slicing the circle into equal $ \sim 2^{-j } $-intervals we have \[ \int_0^{2\pi} |T| = \int_0^{\frac{2\pi}{2^{j-3}}} \sum_{i=0}^{2^{j-3}-1} |T\bigl(x + \frac{i}{2^{j-3}}\bigr)| dx . \] Since $|T(x)| \leq \max_{|u-x| < 2^{2-j}} |F_m(u) - F_m(x)|$, we infer that the sum in the integrand is estimated above by $ \mathbb{VAR}(F_m)$. Taking into account that   $|T| \preceq 1$, we obtain 
		\[ \int_0^{2\pi} |T|^2 \prec \int_0^{2\pi} |T| \leq 2^{-j} \mathbb{VAR}(F_m) \preceq \frac{k_m}{2^j}, \] which is exactly \eqref{Pr3} in the case $ j > m $.
	\end{proof}
\end{proposition}

\begin{proposition} $\left|\widehat{F_m}(n) \right| \prec \frac{\mu_m}{2^{|j-m|}}$ for $2^j \leq |n| < 2^{j+1}$. \label{eq:prop4}
	\begin{proof}
		
		Note that $\int_0^{2\pi} |F_m| \preceq \mu_m$, so it suffices to consider $|m-j| > 10$.
		
		Fix an arbitrary $n$ such that $|n| \in [2^j, 2^{j+1})$. If $j \leq m$ then for an arbitrary $I \in T_m, I = [u,v]$, we have 
		$$ \left|\int_I F_m(x) e^{inx}\right| = \left| e^{inu} \int_I F_m(z)(e^{in(x-u)} - 1) + e^{inu} \int_I F_m(x)\right| \preceq |n| |I|^{2}, $$ 
hence $|\widehat{F_m}(n)| \preceq k_m 2^{j-2m} = \mu_m 2^{ j-m } $, as required.
		
For $j \geq m$, we use \cite[Vol. I. Theorem 4.12]{Zygmund} according to which $ | \hat{f}(n) |\preceq  \mathbb{VAR}(f) /|n| $ for any function $ f $ of bounded variation. Applied to $ f = F_m $ it gives that $|\widehat{F_m}(n)| \le  \mathbb{VAR}(F_m) 2^{-j }  \preceq k_m 2^{-j} =\frac{\mu_m}{2^{j-m}}$, as required.
	\end{proof}
\end{proposition}

\medskip

\textit{Proof of Lemma \ref{T-sec}.} Let $q$ be the dual exponent to $p$. We have 
\[\label{mult} \sum_j \left| P_j(A)\right|_{A_p} \left| P_j(F_i)\right|_{A_q} \underset{\substack{\uparrow \\ \text{H\"older inequality}\\ \text{in each sum in l}}}\ge  
\\ \left|\sum_j \left( \sum_{ |l|= 2^j }^{ 2^{j+1}} \widehat{F_i}(l) \widehat{A}(-l) \right)\right| = |\mathbb{E}(AF_i)| \succeq \frac{\mu_i}{M} . \] 
On account of Propositions 3 and 4
$$ |P_j(F_i)|^q_{A_q} \preceq \left|P_j(F_i)\right|^2_{A_2} \bigl(\max_{2^k \leq |n| < 2^{k+1}} |\widehat{F_m}(n)|\bigr)^{q-2} \preceq \frac{\mu_i}{2^{|i-j|}} \Bigl(\frac{\mu_i}{2^{|i-j|}}\Bigr)^{q-2} . $$ Plugging this into \eqref{mult} one continues the inequality ($ (1 +( q-2 ))/q = 1/p $), 
\[\frac{\mu_i}{M} \preceq 
\sum_j \left| P_j(A)\right|_{A_p} \Bigl(\frac{\mu_i}{2^{|i-j|}} \Bigr)^{1/p }.\] 
 Summing up in $ i $ and using the H\"older inequality, we find
\begin{multline*} \frac{\sum_i \mu_i}M \preceq \sum_{i,j} \left| P_j(A)\right|_{A_p} 2^{-|i-j|/p} \mu_i^{1/p} \le \Bigl( \sum_{i,j} \left| P_j(A)\right|_{A_p}^p 2^{-\tau |i-j|/p} \Bigr)^{ 1/p} \\ \Bigl( \sum_{i,j} \mu_i^{q/p} 2^{-\tau^\prime |i-j|/p} \Bigr)^{1/q} \preceq \left| A \right|_{A_p} \bigl(\sum_i \mu_i^{q/p }\bigr)^{ 1/q } \preceq \left| A \right|_{A_p} \bigl(\sum_i \mu_i \bigr)^{ 1/q } C^{ - \frac rq (\frac qp - 1 )}  \end{multline*}
for some positive $ \tau $, $ \tau^\prime $.

Now, dividing by $\bigl(\sum_i \mu_i \bigr)^{ 1/q }$, and using assumption (3) of Lemma \ref{T-sec} we obtain \[ \left| A \right|_{A_p} \succeq M^{ (2-\alpha)/p -1 } C^{ \frac rq (\frac qp - 1 ) - \frac 1p } \ge C^{ r(\frac 2p - 1) - \frac 1p } ,\]
as required.\end{proof}

	
\section{Dyadic partition}

The idea of the main proof is as follows: we consider the process of dyadic partitioning of the circle into smaller and smaller segments while maintaining certain invariants.

The result of the partitioning algorithm will be a set of certain segments - the vertices of the dyadic partition tree, with the help of which a $T$-sequence will be constructed that satisfies the condition of Lemma \ref{T-sec} for the function $A$, from which the result will follow.

Let us fix the general terminology associated with the dyadic partitioning of the circle.

At each step of the partitioning algorithm, we have a certain set of alive segments (denoted by $\mathcal{AL}$). At each step of the algorithm, we take the longest segment $I \in \mathcal{AL}$, remove it from $\mathcal{AL}$, and either declare it a leaf according to some rules (denote the set of leaves by $\mathcal{L}$), or divide it into 2 segments (called children or sons of $I$) $I_1, I_2$ and put them in $\mathcal{AL}$. Note that the segments $I_1, I_2$ in general may have different lengths.

The segments considered during the algorithm form a set of rooted trees, the roots of which are the set of segments from $\mathcal{AL}$ at the beginning of the algorithm. We will call the father of a segment $I$ such a segment $I'$ that $I$ is a child of $I'$. If segment $I$ has a father, we will denote the father by $\mathcal{F}(I)$. The children of segment $I$, if any, will be denoted by $\mathcal{C}_1(I), \mathcal{C}_2(I)$. We will say that $I$ is an ancestor of $I'$ if $I$ lies on the path from $I'$ to the root of its tree. In this situation, we will also say that $I'$ is a descendant of $I$.

Also, for an arbitrary tree $T$, we will denote the set of its leaves by $\mathcal{L}(T)$.

\section{Characteristic function of a set}

We are going to demonstrate the idea of the proof on a particular case of the problem, which is based on th same ideas and has fewer technical details.

\begin{theorem}\label{main-easy}
	Let $1 < p < \frac{1 + \sqrt{5}}{2}$, $q$ be the dual exponent.
	
	Let $A$ be the characteristic function of an arbitrary measurable set on the circle of measure $\frac{\pi}{C}$, $C$ - a sufficiently large number. Then if for any $m$, $\left|(1 - A)S_m(A) \right|_{A^p} = O\left(\frac{1}{C}\right)$, then $ \left|A \right|_{A^p} \geq C^{\delta}$ for $\delta(p) = \frac{q(2 - p) - 1}{p}$.
\end{theorem}

This theorem is a special case with $\widehat{A} = A$, $\widehat{B} = CA$.

Note without proof that for $p > \frac{1 + \sqrt{5}}{2}$, the condition of this theorem is not satisfied, which can be shown using a slightly modified example by Tselishchev and Lev. The critical case $p = \frac{1 + \sqrt{5}}{2}$ remains open.

\section{Proof of Theorem \ref{main-easy}}

Consider the process of dyadic partitioning of the circle. Initially, we put a single segment equal to the entire circle $[0, 2\pi]$ into $\mathcal{AL}$. We will describe how we process the current segment:

\bfseries Algorithm 1. \mdseries

\begin{enumerate}
	
	\item If $\mathbb{E}_I A \leq \frac{1}{C^2}$ or $\mathbb{E}_I A \geq \frac{1}{10}$, we put $I$ into $\mathcal{L}$ and end the work with the segment. In the former case, we will call this leaf \textbf{small}, and in the latter -- \textbf{large}. We will denote the set of small leaves by $\mathcal{SL}$, large leaves - by $\mathcal{BL}$.
	
	
	\item Otherwise, we remove $I$ from $\mathcal{AL}$, divide it into two equal halves $I_1, I_2$, and put its children $I_1$ and $I_2$ there.
\end{enumerate}

This algorithm can continue infinitely, but the total measure of the leaves is equal to the measure of the entire circle, so we may work with infinite as well.

We will call the \bfseries potential \mdseries of a segment the value $ P_I = (\mathbb{E}_I A)^2 |I|$, and the \bfseries imbalance $D_I$ \mdseries of a segment $ I $ the number 
\[ D_I =|\mathbb{E}_{I_1} A - \mathbb{E}_{I_2} A|, \] where $I_1, I_2$ are the children of the segment.

\begin{lemma} Let $\mathcal{T}$ be an arbitrary set of segments forming a subtree in the dyadic partition tree, such that $I_0$ is its root. Then 
\[\sum_{I \in \mathcal{T} \setminus \mathcal{L}(T)} \frac{I}{4}D_I^2 = \bigl(\sum_{I \in \mathcal{T} \cap \mathcal{L}(T)} P_I \bigr) - P_{I_0} . \]
	
	\begin{proof}
For any segment $ I $ which is not a leaf we have 
\[ \frac 14 D_I^2 = \frac 12 \left( (\mathbb{E}_{ \mathcal C_1 (I) } A )^2 + (\mathbb{E}_{ \mathcal C_2 (I) } A )^2 \right) - (\mathbb{E}_I A )^2  \] 
which is just the identity $ (s-t)^2 = 2 ( s^2 + t^2 ) - ( s+t)^2 $. In the $ P_I $ notation it is written as 
\[ P_{\mathcal{C}_1(I)} + P_{\mathcal{C}_2(I)} - P_I = \frac{I}{4}D_I^2 . \]
Now, when summing up in $ I $, each term $ P_I $ appears in the sum exactly twice, once on its own and once as someone else's son, with the opposite signs, except for the root and the leafs terms. Hence the sum is telescoping to $ (\sum_{I \in \mathcal{T}(I) \cap \mathcal{L}(T)} P_I) - P_{I_0}$, as required.
	\end{proof}
	
\end{lemma} 
 Let's call a segment of the dyadic partition $I$ \textbf{stellar} if one of the following holds:

\begin{enumerate}
	\item $\mathbb{E}_I A \ge \frac{1}{100}$, or
	
	\item $\mathbb{E}_I A \ge \frac{1}{1000}$, and $\mathcal{F}(I)$ is stellar.
\end{enumerate}
In other words, a stellar segment is the one that either itself or one of its ancestors has the average of $ A $ at least $ 1/100 $, and, in the latter case, on all segments on the path from the segment to this ancestor, the average $A$ is at least $\frac{1}{1000}$.

Let's call a stellar segment basic if its father is not stellar. We call a segment final if it is either a leaf or it is not stellar but its father is stellar. Let's denote the set of stellar segments by $\mathcal{SI}$, the set of final segments by $\mathcal{FI}$, and the set of basic segments by $\mathcal{BI}$.

For each segment $I \in \mathcal{BI}$, consider the maximum subtree $S_I$ with root at $I$ such that for any $I_1 \in S_I$, either $I_1 = I$ or $\mathcal{F}(I_1) \in \mathcal{SI}$. Such a subtree is uniquely defined.

\begin{proposition}
	$S_I$ do not intersect for different basic segments.
	\begin{proof}
		Suppose $S_{I_1}$ and $S_{I_2}$ intersect. Then, without loss of generality, $I_2$ lies in the subtree of $I_1$. But then $\mathcal{F}(I_2) \notin \mathcal{SI}$.
	\end{proof}
\end{proposition}

\begin{proposition}
	$\left(\sum_{\mathcal{BL} \cap S_I} P_{I'}\right) - P_I \ge \frac{2}{5000} |I|$.
	\begin{proof}
Notice first that $\frac{1}{50} \ge \mathbb{E}_IA \ge 1/100 $, because $ I $ is stellar (lower bound), and $ \mathcal F ( I ) $ is not (upper bound). In particular, $ P_I \le (1/2500) |I| $. Consider the integral $ \int_{I} A $. This integral is the sum of integrals over three disjoint sets -- the union of small leafs contained in $ S_I $, that of long leafs contained in $ S_I $, and the union of segments from $ S_I $ which are neither leafs, no stellar; the respective integrals to be denoted $(I)$, $(II)$ and $(III)$. Now, $ (I) \preceq O ( C^{-2} ) |I| $, by the very definition of small leafs, $(III) \le (1/1000)|I| $ because $ \mathbb{E}_{I'}A \le \frac{1}{1000}$ for any nonstellar segment $ I^\prime \in S_I $. 

Plugging these into the bound $ \int_{I} A = |I|\mathbb{E}_I A \ge (1/100)|I| $ we conclude that  $ (II) = \int_{\mathcal{BL} \cap S_I} A \ge \frac{98}{1000}|I|$, that is, $ \frac{98}{1000}|I| \le \sum_{J \in \mathcal{BL} \cap S_I } |J| \mathbb{E}_J A $. Since $\mathbb{E}_J A \ge \frac{1}{10}$ for any large leaf $ J $, it follows that 
 $\sum_{J \in \mathcal{BL} \cap S_I} P_{J} \ge \frac{98}{10000}|I|$. Hence $\left(\sum_{J \in \mathcal{BL} \cap \mathcal{L}(S_I)} P_{J}\right) - P_I \ge \frac{2}{5000}|I|$.
	\end{proof}
\end{proposition}

\begin{corollary}
	$\sum_{\mathcal{SI}} |I|D_I^2 \succeq \frac{1}{C}$.
	\begin{proof}\label{stellar}
		Note that each $I_1 \in \mathcal{SI}$ is contained in a unique $S_I, I \in \mathcal{BI}$. Hence \begin{multline*} \sum_{\mathcal{SI}} |I|D_I^2 = \sum_{I \in \mathcal{BI}}\sum_{I_1 \in S_I}|I_1|D_{I_1}^2 = \sum_{I \in \mathcal{BI}} \left[\sum_{I_1 \in \mathcal{BL} \cap S_I} P_{I_1} - P_I \right] \ge \sum_{I \in \mathcal{BI}} \frac{3}{5000}|I| \ge \\ \frac{3}{5000}\sum_{I \in \mathcal{BL}} |I| \succeq \frac{1}{C}. \end{multline*}
		
The penultimate inequality holds because each large leaf is contained in exactly one basic segment, and thus $\sum_{\mathcal{BL}}|I| \le \sum_{\mathcal{BI}}|I|$.
	\end{proof}
\end{corollary}

Let's call the $k$-th level of the tree the set of segments - vertices of the dyadic partition tree having length $\frac{2\pi}{2^k}$ and denote it by $\mathcal{LA}_k$.

\begin{lemma}\label{Feuer} Let $ \mathcal R_k $ be the set of segments $ I $ of $ k $-th level obtained through the Algorithm 1 and such that $ \mathbb{E}_I A > 1/100 $. Then 
$\sum_{I \in \mathcal R_k } |I| \preceq C^{-q} $ with the implied constant independent of $ k $.
\end{lemma}

The specific value $ 1/100 $ in this lemma is of no particular importance, but the implied constant in its assertion would depend on the choice of this value.
	
\begin{proof}
By the assumption of the theorem, $\left| (1-A)S_m(A) \right|_{A^p} = O\left(\frac{1}{C}\right)$. Hence for any $m$, also $ \left|(1-A)(\Phi_m \ast A) \right|_{A^p} = O\left(\frac{1}{C}\right)$, where $\Phi_m = \frac{D_0 + D_1 + \ldots + D_m}{m+1} $ is the Fejer kernel. By the (dual) Hausdorff-Young inequality, $\left\|(1-A)(\Phi_m \ast A) \right\|_{L^q} \preceq \left|(1-A)(\Phi_m \ast A) \right|_{A_p} \preceq \frac{1}{C}$. Let us apply this for $m = 2^k$, and suppose $x \in I$ for $I \in \mathcal{R}_k $. Note that $\Phi_m(y) \ge 0$ $\forall y$ and $\Phi_m(y) \succeq 2^m$ for $y \preceq 2^{-m}$, that is, $(\Phi_m \ast A)(x) \succeq \int_{|y-x| < 2^{-m}} 2^mA(y) \succeq 1$. Notice that the alogrithm implies that $ \mathbb{E}_I A < 1/ 5 $ for any interval $ I $ of the decomposition (otherwise the parent interval would satisfy $ \mathbb{E}_I A > 1/ 10 $, and thus $ I $ shouldn't have appeared), and therefore $ A(x) =0$ for at least $ 4/5 $-th points of the interval $ I $. It follows that    
$\mathbb{E}_I |(1-A)(\Phi_m \ast A)| \succeq 1$, and thus 
\[ C^{-q} \succeq \mathbb{E}|(1-A)(\Phi_m \ast A)|^q \ge  \sum_{I \in \mathcal{R}_k} |I|\mathbb{E}_I|(1-A)(\Phi_m \ast A)|^q \succeq \sum_{I \in \mathcal{R}_k } |I| . \]	\end{proof}

Now we turn to the proof of Theorem \ref{main-easy} properly.

Fix an arbitrary $\alpha$, $2-p > \alpha > 0$. Note that since by Corollary 10 it holds that $\sum_{I \in \mathcal{SI}} |I|D_I^2 \succeq \frac{1}{C}$, then for some $M$ it is true that $\sum_{I \in \mathcal{SI}, \frac{1}{M} \ge D_I \ge \frac{1}{2M}} |I|D_I^2 \succeq \frac{1}{CM^{\alpha}}$. Take this $M$ and call the stellar segments with $\frac{1}{M} \ge D_I \ge \frac{1}{2M}$ interesting. Denote the set of such segments by $\mathcal{IS}$.

Now define the function $F_i$ as follows:

\begin{enumerate}
	\item $F_i(x) = 0$, $x \notin \mathcal{LA}_i \cap \mathcal{IS}$.
	
	\item Let $I \in \mathcal{LA}_i \cap \mathcal{IS}$, $I_1, I_2$ - its children. Then if $\mathbb{E}_{I_1}A > \mathbb{E}_{I_2}A$, then $F_i(x) = 1$, $x \in I_1$, $F_i(x) = -1$, $x \in I_2$. Otherwise $F_i(x) = -1$, $x \in I_1$, $F_i(x) = 1$, $x \in I_2$.
\end{enumerate}

It is easy to see that $\mathcal{TS}_i = (F_i, \mathcal{LA}_i \cap \mathcal{IS})$ is a $T$-system. Let us verify that Lemma 2 is satisfied for these numbers $p, r=q, \alpha, M, C$, as well as the $T$-sequence $\mathcal{P} = {\mathcal{TS}_0, \mathcal{TS}_1, \mathcal{TS}_2, \ldots}$ and function $A$.

\begin{enumerate}
	\item It is satisfied since $\mathbb{E}(F_i A) = \sum_{I \in \mathcal{LA}_i \cap \mathcal{IS}} |I| \mathbb{E}(F_i A) \succeq \frac{\mu_i}{M}$.
	
	\item Satisfied by Lemma \ref{Feuer}.
	
	\item $\sum_{I \in \mathcal{IS}} |I|D_I^2 \succeq \frac{1}{CM^{\alpha}}$ by the definition of interesting segments, hence $\sum_i \mu_i = \sum_{I \in \mathcal{IS}} |I| \succeq \frac{M^{2-\alpha}}{C}$.
\end{enumerate}

Applying Lemma \ref{T-sec} now, we get that $ \left|A \right|_{A_p} \succeq C^{\frac{q(2-p)-1}{p}}$, which completes the proof of Theorem \ref{main-easy}. 
	
	\section{Auxiliary Results}
	
The proof of Theorem \ref{main} follows the same scheme as the one of Theorem \ref{main-easy}, but more auxiliary results are needed.
	
Denote $A = \Re \widetilde{A}, B = \Re \widetilde{B}$. Given a segment $I \subset [0, 2\pi]$, we set
	\[ IA(I) = \min  \left( |I|, \left| \int_I A \right| \right). \]
	
Let $I_1, \dots, I_k$ be a disjoint set of subintervals of the interval $I$. The weight of the set is defined as $\sum_{i=1}^k IA(I_i)$. The $A_0$-measure of a segment $I$ of the circle is defined as
$	 A_0(I) := \sup_{(I_1, \dots, I_k)} \sum_i IA(I_i), $
	where the supremum is taken over all disjoint sets of subintervals. Some trivial properties of the $ A_0 $-measure are collected in the following
	
	\begin{proposition}\label{A0_prop}\begin{enumerate}
			\item If $a < b < c$, then $A_0([a, c]) \ge A_0([a, b]) + A_0([b, c])$;
			
			\item $\min{ \int_I A, |I| } \le A_0(I) \le |I|$;
			
			\item $\int_I \min(|A|, 1) \le A_0(I) \le \int_I |A|$.
		\end{enumerate}
	\end{proposition}
	
	The $A_0$-measure of a set of non-intersecting segments is defined as the sum of the $A_0$-measures of the segments in it.
	
	Let $t \in (0, 1)$. We define the \textbf{$t$-norm of discrepancy of the interval $I$} as follows,
	\[
	 \mathcal{DA}(I, t) := \sup_{I_1 \subset I: \, t|I| \le |I_1| \le |I|} \left| \mathbb{E}_{I_1}A - \mathbb{E}_I A \right|. \]
	In the situation when $t = \frac{1}{10}$, the argument $t$ is omitted, and the quantity $\mathcal{DA}(I, t)$ is called simply the norm of discrepancy and denoted by $\mathcal{DA}(I)$. Clearly the supremum in this definition is attained.
	
	\begin{proposition}
If $t_1 < t_2$, then $\mathcal{DA}(I, t_1) \ge \mathcal{DA}(I, t_2)$. If $t < \frac{1}{10} $ then $ \mathcal{DA}(I, 2t) \ge \frac{1}{5} \mathcal{DA}(I, t)$.
\end{proposition}
		
		\begin{proof}
The first inequality is true since the supremum is taken over a larger set. To prove the second, denote by $I'$ the interval on which the maximum in $\mathcal{DA}(I, t)$ is attained. Withot loss of generality, we assume that $ \mathbb{E}_I A = 0 $. If $|I'| \ge 2t|I|$, then $\mathcal{DA}(I, 2t) = \mathcal{DA}(I, t)$. If not, consider an arbitrary interval $ I_1 \subset I $ adjacent to $  I^\prime $ of the length\footnote{$ t < 1/10$ ensures the existence of such an interval within $I$.} $ |I_1 | = 2 |I^\prime | $. Then $|\mathbb{E}_{I_1\cup I^\prime }A | \le \mathcal{DA}(I, 2t) $, $ |\mathbb{E}_{I_1}  A | \le \mathcal{DA}(I, 2t)$, and $ 3 \mathbb{E}_{ I_1\cup I^\prime } A = 2 \mathbb{E}_{ I_1} A +  \mathbb{E}_{I^\prime} A $ hence \[ \mathcal{DA}(I, t ) = |\mathbb{E}_{I^\prime }  A | \le 3|\mathbb{E}_{I_1\cup I^\prime }A | + 2|\mathbb{E}_{I_1}A | \le 5\mathcal{DA}(I, 2t) . \]
\end{proof}

	\begin{proposition}\label{A0DA}
		\begin{equation} \label{1100} A_0(I) \ge |I| \min\left(\frac{1}{100}, \frac{1}{100} \mathcal{DA}(I)\right).
		\end{equation}
			\end{proposition}
		
		\begin{proof}
Arguing by contradiction, suppose \eqref{1100} is false. Then $|\mathbb{E}_I A| < \min\left(\frac{1}{100}, \frac{1}{100} \mathcal{DA}(I)\right)$, otherwise take the entire interval as a partition for $A_0$.
			
			Take in the definition of $A_0$ measure one of the segments of the partition that achieves the maximum in $\mathcal{DA}$ (call it $I'$). Then $|\mathbb{E}_{I'}A| \ge \frac{99}{100} \mathcal{DA}(I)$. But then $$A_0(I) \ge |I'| \min(1, |\mathbb{E}_{I'}A|) \ge |I| \min\left(\frac{1}{10}, \frac{99}{1000} \mathcal{DA}(I)\right) , $$
			a contradiction. 
\end{proof}
	
Given a set $U$, a probability measure $\mu$ on it, a real measurable function $F$ on $U$, and fixed real numbers $q_1 \ge 1, c > 0 $, we call a real number $y $ a \textbf{$(q_1, c)$-almost midpoint} if for all $t$ the inequality $\int_U |F - t|^{q_1} d\mu \ge c|y - t|^{q_1}$ holds. By the H\"older inequality
	
	\begin{remark}\label{almostmid}
		For any measurable $ \Sigma \subset U $ the number $\mu \left( \Sigma \right)^{ -1 }\int_\Sigma F d\mu $ is a $(q_1, 1/ \mu ( \Sigma ) )$-almost midpoint for any $ q_1 \ge 1 $.
	\end{remark}
	
We omit the constant $ c $ in this definition if it can be chosen uniformly with respect to parameters of the problem and call the number $y$ simply a $q_1$-almost midpoint.
	

For an interval $I$ let
\[ \mathcal{U}(I) = \bigl\{ (a, b): a, b \in I, b > a, |a - b| > \frac{|I|}{100} \bigr\} ,\] 
and let $\mu_U$ be the planar Lebesgue measure on this set normalized to $ 1 $, that is, $ d \mu_U = \frac{C{\mu}}{|I|^2} da db $ for some $ C_\mu $. In fact, one easily verifies that $C_{\mu} = 2\frac{100^2}{99^2}$, but we do not need the specific value.
	
Given a differentiable function $T$ on the interval $I$, let 
\[ G_{I, T} (a, b ) := \frac{T(b) - T(a)}{b - a}, \; (a,b) \in \mathcal{U}(I) . \]
	
Given an interval $I$ and a differentiable function $T$ on it, a real number $x$ is \bfseries $(q_1, I, T)$-balanced \mdseries if it is a $q_1$-almost midpoint for the set $U = \mathcal{U}(I)$, the measure $\mu_U$ on it, and the function $G_{I, T}$.
	
Finally, let us define a \textbf{balanced point}. For an interval $I$, a continuously differentiable real function $T$ on it, and a real number $q_1 \ge 1$, we call a point $t$ of the interval $I$ balanced for $T, I$ if for all $x$, $0 < x \le |I|$, there exists an interval $ I(x) \subset I $ such that
	
	\begin{enumerate}	
	\item $t \in I(x)$,
	
	\item $\frac{x}{100} \le |I(x)| \le x $,
	
	\item The number $T'(t)$ is $(q_1, I(x), T)$-balanced.
	
\end{enumerate}
	
 \begin{lemma}\label{midpoint} For any continuously differentiable $T$ within any interval $I$, there exists a balanced point for $T, I$. Moreover, for a fixed $|I|$, there exists a specific non-negative function $f(x,y)$ on the set 
$\mathcal{U}(I)$ with mean value $1$, bounded by an absolute constant independent of $|I|$, such that there exists a balanced point $t_0$ such that $T'(t_0) = \mathbb{E}_{U} f(a,b) \frac{T(b) - T(a)}{b - a}$, and this does not depend on the function $T$.
\end{lemma}

\begin{proof}
Given a segment $ J_1 = J $ define $ Y_\pm = \mathbb{E}_{ \mathcal M_\pm  (J) } G_{ J, T}$, $ Y = \mathbb{E}_{ \mathcal M_- (J) \cup \mathcal M_+ (J) } G_{ J, T} $ with the sets $ \mathcal M_\pm ( J ) $ chosen as shown in the picture, that is, $\mathcal M_- (J) = \{ (s,t) \in \mathcal U (J): s-a< 0.48 |J| ; \, 0.04 |J| < t-a < 0.52 |J| \} $, and $ \mathcal M_+ (J) \subset \mathcal U (J)  $ is a translation of the triangle $ \mathcal M_- (J) $ along the line $ s = t $ such that the interiors of $ \mathcal M_- (J) $ and $ \mathcal M_+ (J)  $ do not intersect, and the average is taken with respect to the Lebesgue measure. 

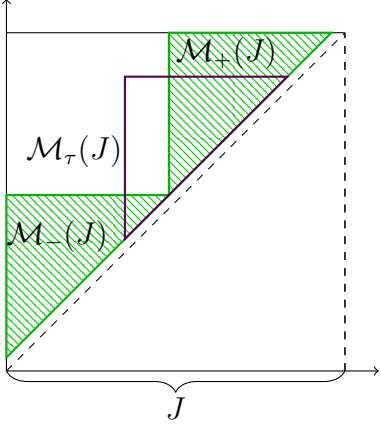
\begin{tikzpicture}[scale=4.5]
  \draw[->] (0,0) -- (1.1,0);
  \draw[->] (0,0) -- (0,1.1);


  \draw[dashed] (0,0) -- (1,1);
\draw[dashed] (1,0) -- (1,1);
\draw (0,1) -- (1,1);

  \draw[dashed] (1,0) -- (1,1);

\draw[green!70!black, thick] (0.48,0.52) -- (0.48,1) -- (0.96,1) -- cycle;
\begin{scope} \clip (0,0) -- (1,1) -- (0,1) -- cycle;  
\fill[green!80,pattern=north west lines,pattern color=green!80!black] (0.48,0.52) -- (0.48,1) -- (0.96,1) -- cycle; \end{scope} 
\draw[green!70!black, thick] (0,0.04) -- (0,0.52) -- (0.48,0.52) -- cycle;
\begin{scope} \clip (0,0) -- (1,1) -- (0,1) -- cycle;  
\fill[green!80,pattern=north west lines,pattern color=green!80!black] (0,0.04) -- (0,0.52) -- (0.48,0.52) -- cycle; \end{scope} 
\draw[violet!70!black, thick] (0.35,0.39) -- (0.35,0.87) -- (0.83,0.87) -- cycle;
\begin{scope}  \clip (0,0) -- (1,1) -- (1,0) -- cycle; 
\fill[violet!80,pattern=north west lines,pattern color=violet!80!black] (0.35,0.39) -- (0.35,0.87) -- (0.83,0.87) -- cycle; \end{scope} 

  \node at (0.65,0.94) {$\mathcal{M}_+ (J)$};
  \node at (0.15,0.4) {$\mathcal{M}_- (J)$};
  \node at (0.2,0.65) {$\mathcal{M}_\tau (J)$};

  \draw [decorate,decoration={brace,mirror,amplitude=8pt}] (0,0) -- (1,0) node[midway, yshift=-0.5cm] { $J$};
\end{tikzpicture}

Let $ \mathcal M_\tau  (J) = \mathcal M_-  (J) + (1,1) \tau $, $ \tau \in [\tau_- , \tau_+  ] \equiv [0, 0.48] $. Clearly  $ \mathcal M_\tau  (J)\subset \mathcal U (J) $,  $ Y ( \tau_\pm )  = Y_\pm $, $ Y = \frac{ Y_- + Y_+ }2 $, and the function $ Y ( \tau )  = \mathbb{E}_{ \mathcal M_\tau  (J) } G_{ J, T}$ is continuous in $ \tau $. By the continuity, there exists a $ \tau \in  [\tau_1, \tau_2 ]$ such that $ Y(\tau ) = Y $. Define the new segment $ J_2 $ to be the projection of $ \mathcal M_\tau (J )$ (the violet triangle in the picture) on the $x$ axis.  

Iterating this procedure we obtain a sequence of nested segments $ J_k $, $ | J_{ k+1 } | \le \frac 23| J_k | $. Let $ t_0 $ be their (unique) common point, $ t_0 = \cap_k J_k $. Let us check that this point verifies all the assertions of the lemma. 

First, $ Y $ is a $ ( q_1, 1/4 )$-almost midpoint for the function $ G_{ J , T } $ for each $ k $ by remark \ref{almostmid}. By the mean value theorem, for each $ k $ there exist points $ \xi_k , \eta_k \in J_k $ such that $ Y =  G_{I, T} (\xi_k, \eta_k ) $. In turn, $ G_{I, T} (\xi_k, \eta_k ) = \frac{T( \xi_k ) - T(\eta_k )}{\xi_k  - \eta_k } = T^\prime ( s_k ) $ for some $ s_k \in J_k $ by elementary calculus. The continuity of $ T^\prime $ implies that $ T^\prime ( s_k ) \longrightarrow T^\prime (t_0) $, hence $ T^\prime (t_0) = Y $. Then, the interval $ I (x) $ is taken to be the one of $ J_k $ for which $ | J_k | \le x \le | J_{ k-1} | $. Finally, the universal function $ f $ is just the indicator of $ \mathcal M_- (J) \cup \mathcal M_+ (J)  $. \end{proof}

Now we are ready to present the dyadic decomposition algorithm for proving the main theorem.

\section{Dyadic Decomposition Algorithm}

The decomposition algorithm will proceed through phases, and in the $i$-th phase, it will be determined by four real numbers $K_i > 1, C_i > 1, 0 < r_i < 1, 0 < \Delta_i < 1$. In the description of a specific phase, indices will sometimes be omitted. These numbers are set as follows:

\begin{enumerate}
\item $K_1 = 10^{10}(1 + \frac{1}{\delta}), C_1 = \frac{1}{A_0([0, 2\pi])}, r_1 = \frac{1}{1000}, \Delta_1 = 10^{-10}$
\item $K_{i+1} = \frac{11}{10}K_i, C_{i+1} = K_iC_i, r_{i+1} = r_i + \Delta_i, \Delta_i = \frac{1}{K_i}$ for $i > 1$.
\end{enumerate}

Note that by Proposition \ref{A0_prop}, $C_1 \ge \frac{1}{|A|_{L_1}} \ge C$.

\begin{proposition}
Let $F$ be a function on the circle. Then $|F|_{A^p} \ge \frac{|F|_{L_1}}{2\pi}$.
\begin{proof}
	$|F|_{A^p} \ge \frac{|F|_{L_q}}{2\pi} \ge \frac{|F|_{L_1}}{2\pi}$, where the first inequality follows from the Hausdorff-Young inequality.
\end{proof}
\end{proposition}

\begin{proposition}
For any $\beta > 0$, there exists a $D > 0$ such that if $A_0([0, 2\pi]) < \frac{1}{D}$, then for all $i$, $C_i^{\beta} > K_i, C_i^{\beta} > \frac{1}{\Delta_i}$. Also, $r_i < \frac{1}{10}$ for all $i$.
\end{proposition}

We introduce the notation $F \succeq_K G$ for functions $F, G$, which means that $\exists$ fixed real numbers $C_1, C_2, C_3 \ge 0$ such that $C_1K^{C_2}\Delta^{-C_3}F \ge G$. Other notations with the subscript $K$, such as $F = O_K(G)$, have the same meaning.

Each phase of the algorithm represents a finite process of dyadic decomposition of a tree starting from some set of initial living intervals. It can have two different results - success and failure. In case of success, we will be able to choose some subset of intervals of the decomposition tree and apply Lemma 2 to it. In case of failure, we will choose some new set of intervals, which will become the initial living intervals for the next phase.

At each phase, we will place some subintervals of the leaves of the tree into three different sets, to be called respectively $A$-small set ($\mathcal{AS}_i$), $B$-small set ($\mathcal{BS}_i$) and final set ($\mathcal{FS}_i$). By $\mathcal{NL}_i$, we denote the set of intervals of the decomposition tree at the $i$-th phase that are not leaves. The previously defined notations for leaves, living intervals, father of an interval, etc., will be used as well.

We denote by $S_a(x)$ a periodic function equal to $\int_a^x B(x)$ for $x < a + 2\pi$. By $T_a(x)$ we denote a periodic function equal to $\int_a^x B(x)$ for $x < a + 2\pi$.

For a real $c > 0$, we call a point of the circle $y$ $c$-balanced if there exists an interval $I$ such that $c > |I| > \frac{c}{100}$ and for any point on the circle $v \notin I$, $y$ is balanced for $T_v$ on $I$ with $q_1 = q = \frac{p}{p-1}$.

Note that this definition is sufficient to check for any one such point $v$, because all other such $T_{v_1}$ differ from $T_v$ by some linear function.
	
	\begin{proposition}
		In any segment of the circle $I$, $|I| < 2\pi$, by Lemma 15, there exists a $|I|$-balanced point.
	\end{proposition}
	
At the beginning of the $i$-th phase, the set of living segments is either $\mathcal{AS}_{i-1}$ for $i > 1$, or the entire circle for $i = 1$. In the circle, we will choose the "beginning" as follows: we select the segment of the circle $\left[-\frac{\pi}{2}, \frac{\pi}{2}\right]$, and on it, by Proposition 20, we choose a $\pi$-balanced point $t$, which we will take as the "beginning". For convenience, we will consider this point to be 0, and also denote $S = S_0$, $T = T_0$.
	
	The algorithm proceeds in the same way as in the proof of Theorem 2, namely, we will take the longest segment $I \in \mathcal{AL}$, remove it from $\mathcal{AL}$, and either declare it a leaf according to some rules, or divide it into 2 segments and place them in $\mathcal{AL}$.
	
	We will now describe how a segment is processed at the $i$-th phase. The index $i$ will now be omitted.
	
	\bfseries Algorithm 2. \mdseries 
	
	Recall that $A_0(I)$ is the $A_0$-measure, $DA(I, t)$ is the norm of discrepancy, and $\mathcal{F}(I)$ is the father of the segment in the tree. If a segment fits several points of the algorithm, then the point with the smallest number is applied.
	\begin{enumerate}
		
		\item If $\int_I B > 2KCA_0(I)$, then declare $I$ a leaf and place $I$ in $\mathcal{AS}$.
		
		\item If $K \int_I B < CA_0(I)$, then declare $I$ a leaf and place in $\mathcal{BS}$.
		
		\item If $DA(I, \frac{1}{10000K^{10}}) \ge \frac{1}{10000K^5}$, then declare $I$ a leaf and place in $\mathcal{FS}$.
		
		\item If $|\mathbb{E}_I A| \ge r + \frac{\Delta}{2}$, then declare $I$ a leaf and place in $\mathcal{FS}$.
		
		\item Call a segment dense if its $A_0$-measure is at least $\frac{|I|}{100}$. If the segment $I$ is dense and there exists a subsegment $I'$ of length $\frac{|I|}{K^2}$ such that $\int{I'} B \ge \left(1 - \frac{3}{K}\right)\int_I B$, and its ends are $|I|'$ -- balanced, then declare $I$ a leaf and place $I'$ in $\mathcal{AS}$.
		
		\item If none of the previous points are satisfied, take a $\frac{|I|}{3}$-balanced point $t$ for the middle third of the segment, which exists by Lemma 15, divide the segment with it into 2 subsegments, and place them in $\mathcal{AL}$.
		
	\end{enumerate}
	
	Declare the decomposition phase successful if $\sum_{I \in \mathcal{FS}} A_0(I) \ge \frac{1}{K^{3}C}$, and end the entire decomposition algorithm in this case. Otherwise, declare it unsuccessful, take $\mathcal{AL}$ as the set of living segments for the next phase, and move on to the next phase.
	
	
	\begin{proposition}\label{BeqA}
		$\forall I \in \mathcal{FS} \cup \mathcal{NL}$, $\frac{CA_0(I)}{K} \le \int_I B \le 2KCA_0(I)$.
		\begin{proof}
			If these inequalities are not satisfied, then the segment $I$ will be processed by points 1 and 2 and placed in other sets of leaves.
		\end{proof}
	\end{proposition}
	
	\begin{proposition}
		At any given time during the algorithm, $\forall I \in \mathcal{AL}_i$, both ends of $I$ are $\frac{|I|}{10}$-balanced points.
		\begin{proof}
			If these inequalities are not satisfied, then the segment $I$ will be processed by points 1 and 2, and in the case of point 1, we will place the father of $I$ in $\mathcal{FS}$, which satisfies these inequalities.
		\end{proof}
	\end{proposition}
	
 The given properties are checked directly:
\begin{proposition}
	Throughout the entire algorithm, the following invariants will be satisfied:
	\begin{enumerate}
		\item At the beginning of the $i$-th phase, $\forall I \in \mathcal{AL}$, $\int_I B \ge C_i A_0(I)$.
		\item At the beginning of the $i$-th phase, $\int_{\mathcal{AL}_i} B \ge 2\pi - \delta$.
		\item At any moment in time, for $\forall I \in \mathcal{AL}_i$, $DA(I, \frac{1}{100K_i^2}) \le \frac{1}{100K_i^3}$.
		\item At the beginning of the $i$-th phase, for $\forall I \in \mathcal{AL}_i$, $|\mathbb{E}_I A| < r_i$.
	\end{enumerate}
\end{proposition}

\begin{proof}
	Exercise for the reader.
\end{proof}

\begin{proposition}
The number of phases of the algorithm is finite.
\begin{proof}
	Suppose the contrary and denote by $\mathcal{T}_k$ the set of alive segments at the beginning of the $k$-th phase. Then $\mathcal{T}_k$ is a sequence of nested sets, denote their intersection by $\mathcal{T}$. $|\mathcal{T}| > 0$, since $\forall k$, $\int_{\mathcal{T}_k} B > 2\pi - \delta$ by invariant 2 of the previous proposition. On the other hand, $\sum_{I \in \mathcal{T}_k}A_0(I) \preceq \frac{1}{C_i}\int_{T_k}B \preceq C_i \int |B| \to 0$, hence by Proposition \ref{A0_prop}, $\sum_{I \in \mathcal{T}k} \int_I \min(|A|, 1) \to 0$, which means $A(x) = 0$ almost everywhere on $T$, and $\int |B(1-A)| \ge \int_{\mathcal{T}} |B(1-A)| = 2\pi - \delta$, that is, $\left| B(1-A)\right|_{A^p} \ge 1 - \frac{\delta}{2\pi}$, which contradicts condition 4 of the theorem.
\end{proof}
\end{proposition}

From now on we assume that the current phase is successful, since we have shown that a successful phase will necessarily occur.

Denote by $T$ the set of segments of the dyadic partition tree in this phase. Also denote
 $\mathcal{BAL}(I) = |I| DA(I, \frac{1}{10^{100}K^{100}})^2$.

\begin{lemma}\label{SumBal}
$\sum_{I \in \mathcal{NL}, A_0(I) \ge \frac{|I|}{10^{1000} K^{1000}}} \mathcal{BAL}(I) \succeq_K 1$
\begin{proof}
	Leaves $I \in \mathcal{FS}$ are divided into two types - such that $\mathbb{E}I \ge r + \frac{\Delta}{2}$ (denote the set of such by $\mathcal{FSB}$), and the remaining ones. $\sum_{I \in \mathcal{FS}} A_0(I) \ge \frac{1}{K^3C}$, hence either $\sum_{I \in \mathcal{FS \setminus FSB}} A_0(I) \ge \frac{1}{2K^3C}$, or $\sum_{I \in \mathcal{FSB}} A_0(I) \ge \frac{1}{2K^3C}$. In the first case, note that $$DA(\mathcal{F}(I), \frac{1}{10^{100}K^{100}}) \ge \frac{1}{10^5K^5}$$ for  $\forall I \in \mathcal{FS \setminus FSB}$, whence $A_0(\mathcal{F}(I)) \ge \frac{|I|}{10^{1000} K^{1000}}$ and $$|\mathcal{F}(I)| DA^2(\mathcal{F}(I), \frac{1}{10^{100}K^{100}}) \succeq_K |I|$$ by \ref{A0DA}, whence $\sum_{I \in \mathcal{FS \setminus FSB}} \mathcal{BAL}(\mathcal{F}(I)) \succeq_K 1$, but in this sum each segment occurs at most 2 times, whence $\sum_{I \in \mathcal{NL}, A_0(I) \ge \frac{|I|}{10^{1000} K^{1000}}} \mathcal{BAL}(I) \succeq_K 1$.
	
	In the case if $\sum_{I \in \mathcal{FSB}} A_0(I) \ge \frac{1}{2K^3C}$, the proof repeats the proof of Proposition \ref{stellar}.
\end{proof}
\end{lemma}

To complete the proof of the theorem, we need an analogue of Lemma \ref{Feuer}.


\section{Skewed Segments}

A \textbf{skewed segment} (SST) $\eta = \eta(a,b,c,d)$ for real $a < b < c < d$ will be called a function on the axis (or circle identified with $ [-\pi, \pi ] $) equal to $0$ outside the segment $[a,d]$, $1$ on the segment $[b,c]$, and linear on the segments $[a,b]$ and $[c,d]$, see the picture.

\bigskip
\medskip

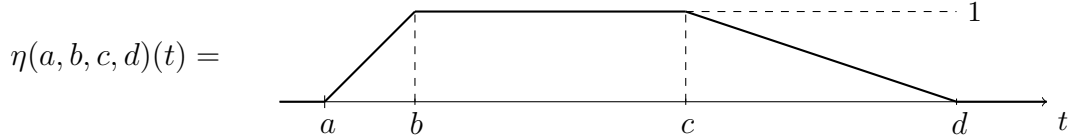
\begin{tikzpicture}[scale=1.2]\label{trpezoid}
\draw[xshift=-2.6cm,yshift=.5cm]
node[right,text width=6cm]
{$\eta ( a,b,c,d) (t) = $};
\draw[->] (0.5,0) -- (9,0) node[anchor=north west]{\small $t$};
\draw[thick] (0.5,0) -- (1,0);
\draw (1,2pt) -- ( 1,-2pt) node[anchor=north west,pos=1pt]{$\!\!\!a$};
\draw[thick] (1,0) -- (2,1);
\draw (2,1pt) -- ( 2,-1pt)
node[anchor=north west,pos=0.4pt]{$\!\!\!b$};
\draw[dashed] (2,0) -- (2,1); 
\draw[thick] (2,1) -- ( 5,1);
\draw[dashed] (5,0) -- (5,1); 
\draw (5,1pt) -- ( 5,-1pt)
node[anchor=north west,pos=1.5pt]{$\!\!\!c$};
\draw[thick] (5,1) -- (8,0);
\draw (8,1pt) -- (8,-1pt) node[anchor=north west,pos=0.4pt]{$\!\!\!d$};
\draw[dashed] (5,1) -- (8,1) node[anchor=west]{\small $1$};
\draw[thick] (8,0) -- (9,0);
\end{tikzpicture} 

\medskip

A skewed segment is called \textbf{symmetric} (SSST) if $b-a = d-c$, in this case the point $\frac{b+c}{2}$ is called the center of the segment. An SST is called \textbf{proportional} if $\frac{|b-a|}{|d-a|}, \frac{|d-c|}{|d-a|} > \frac{1}{1000}$. The integral of a function $f$ over a skewed segment will be denoted by $\int_{\eta} f = \int_a^d f \eta dx$, the average of $f$ over a skewed segment is $\mathbb{E}(a,b,c,d)f = \frac{\int_a^d f \eta}{\int_a^d \eta}$. Also, for a function $f$ on a circle, denote by $\mathcal{SK}_{a,b}(f)$ the function equal to $\mathbb{E}(x-a,x-b, x+b, x+a)f$ at point $x$, where $\eta(a,b,c,d)$ is understood as a function on the circle. By definition, $\mathcal{SK}_{a,b}(F) = \frac{2}{d+c-b-a} F \ast \eta(-a,-b,b,a)$. The segment $[b,c]$ will be called the middle of the SST.

Let $\mathcal{D}(a,b,c,d) ( t ) = \frac{ d\eta (a,b,c,d)(t)}{dt} $, $ t \ne a, b, c, d $. 


Through $\mathcal{SR}_\beta$ denote the set of such SSST $\eta(a,b,c,d)$ that $\frac{b-a}{d-c} = \beta$.

The set of all symmetric segments with center at $t$ will be denoted by $\mathcal{SYMS}_t$. For real $x,y$, through $\mathcal{PRS}_{(x,y,z)}$ denote the set of SST $\eta(a,b,c,d)$ such that $x \le b-a \le z, x \le d-c \le z, y \ge d-a$. 

Let us denote $$FB_{a,b,c}(z) = \frac{1}{a^2} \int_{ a < d^\prime - c^\prime < c; 0 < 2d^\prime < b } |\mathbb{E}(z-d^\prime ,z-c^\prime ,z+c^\prime ,z+d^\prime )B|^q dc^\prime dd^\prime .$$

For a segment $I$, define $$\mathcal{SSI}(I) = \int_I FB_{\frac{|I|}{2^{400}}, 2^{400}|I|,2^{400}|I|}(z) dz$$ and $$\mathcal{SAI}(I) = \int_I FB_{\frac{|I|}{2^{600}}, 2^{600}|I|, 2^{600}|I|}(z)|1-A(z)|^q dz.$$

Also for a segment $I$, define the \textbf{external skewed set} $OU(I)$ as follows:

$$OU(I) = {\eta(a,b,c,d): \eta(a,b,c,d) \in \mathcal{PRS}_{\frac{|I|}{2^{300}},\frac{(2^{200}+1)|I|}{2^{200}}, \frac{|I|}{2^{200}}}, I \subset [b, c] } $$

Also define $FI(I) = \frac{1}{|I|^4} \int_{\eta(a,b,c,d) \in OI(I)} \left|\mathbb{E}(a,b,c,d)B\right|^q$.

\begin{proposition}\label{Fouriercoeff}
If $ F $ is a 	symmetric skewed segment (considered as a function on the circle) with $ d \asymp d-c $, then the total variation of its Fourier coefficients $ \hat{F}(n) $ is bounded by $ d  $,  $\sum_{k=-\infty}^{\infty} |\widehat{F}(k+1) - \widehat{F}(k)| \preceq d$.
\end{proposition}

\begin{proof}
Integration by parts gives $ |\hat{F} (n)| \preceq (d-c)^{-1}   n^{-2} $. Thus, the sum of $  |\widehat{F}(k+1) - \widehat{F}(k)| $ over $ |k| > (d-c)^{ -1 } $ is trivially $ \preceq d $. On the other hand,  \[ |\widehat{F}(k+1) - \widehat{F}(k)| \le \int_k^{k+1} \Bigl|\frac{d\widehat F(s)}{ds}\Bigr| ds \le \max_{ s \in \mathbb R } \Bigl|\frac{d\widehat F(s)}{ds}\Bigr|  \le \int_{-d}^d |x F ( x )| dx \preceq d^2 \] as $ |F | \le 1 $.  Now the sum of $  |\widehat{F}(k+1) - \widehat{F}(k)| $  over $ |k| \le (d-c)^{ -1 } $   is $ \preceq d^2/(d-c) \preceq d  $
by the assumption $ d \asymp d-c $, and the result follows. 
\end{proof}


\begin{corollary}\label{convol}
	Let $\eta(-d, -c, c, d)\in \mathcal{PRS}_{\frac{|I|}{2^{600}}, 2^{600}|I|, 2^{600}|I|}$ , $d \preceq 1$. We will treat $\eta(a, b, c, d)$ as a function on the circle. Then there exist real numbers $\alpha_0, \alpha_1, \ldots$ such that $\sum |\alpha_i| \preceq d$ and 
\begin{equation}\label{Four} F \ast \eta(-d, -c, c, d) = \sum_{i=0}^{\infty} \alpha_i S_i(F) \end{equation} 
 for all trigonometric polynomials $ F $. 	
\end{corollary}
	
\begin{proof}
Let $\alpha_i = \widehat{\eta(-d, -c, c, d)}(i) - \widehat{\eta(-d, -c, c, d)}(i+1)$. Identity \eqref{Four} holds since  the right and left hand sides of it have the same Fourier coefficients,  the $k$-th Fourier coefficient being $ \sum_{i=k}^{\infty} \alpha_i \widehat{F}(k)$. Finally $\sum |\alpha_i| \preceq d$ by Proposition \ref{Fouriercoeff}. 
\end{proof}

\begin{corollary}\label{maxS}
	Let $G \in A_1$, $1 \leq a > b > 0, 2^{20}(a-b) > b$, and $\mu$ be a finite measure on the circle. Then 
\[\int |\mathcal{SK}_{a,b}(G)|^q d\mu \leq \max_n \int \left| S_n(G) \right|^q d\mu . \]  \end{corollary}

	\begin{proof}
In the current notation, $\mathcal{SK}_{a,b}(G) = \frac{1}{a+b} G \ast \eta(-a, -b, a, b)$. Since $S_n(G) \to G$ uniformly, it suffices to prove that 
\[ \int \left| S_n(G) \ast \eta (-a,-b,b,a) \right|^q d\mu \preceq a^q \max_{ i \le n } \int \left|S_i(G) \right|^q d\mu\] 
uniformly in $ n $. By Corollary \ref{convol}, $S_n(G) \ast \eta(-a,-b,b,a) = \sum_{i=0}^n \alpha_i S_i(G)$, hence by the triangle inequality in $ L^q ( d\mu ) $ we have $$\int \left| S_n(G) \ast \eta (-a,-b,b,a) \right|^q d\mu \preceq (\sum |\alpha_i|)^q \max_{i\le n } \int |S_n|^q d\mu \preceq a^q \max_n \int |S_n|^q d\mu $$ as required.
\end{proof}

\begin{lemma}\label{Disj1}
	Let a disjoint set of segments $\mathcal{M}$ be chosen on the circle, such that the lengths of any two of them differ by no more than a factor of 2. Then $\sum_{I \in \mathcal{M}} \mathcal{SAI}(I) \preceq 1$.
\end{lemma}

\begin{proof}
	Denote the length of any of the segments $\mathcal{M}$ by $l$. Let us also denote by $\mathcal{P}_l$ the set of such pairs $(a, b)$ that $\eta(-a, -b, b, a) \in \mathcal{PRS}_{\frac{l}{2^{1000}}, 2^{1000}l, 2^{1000}l}$. Then
\begin{multline*} \sum_{I \in \mathcal{M}} \mathcal{SAI}(I) \preceq \frac{1}{l^2} \int_{\mathcal{P}_l} \int_{y \in \mathbb{T}} |\mathbb{E}(y-a, y-b, y+b, y+a)B|^q |1-A(x)|^q dx \preceq \\
\preceq \frac{1}{l^2} \int_{(a, b) \in \mathcal{P}_l} \int |\mathcal{SK}_{a,b}(B)(x)|^q |1-A(x)|^q dx \preceq \frac{1}{l^2} \int_{(a, b) \in \mathcal{P}_l} \max_n \left\| S_n(B)(1-A) \right\|_{L_q}^q
\end{multline*}  by Corollary \ref{maxS}, since $\int |1-A|^q < \infty$. But by the Hausdorff-Young inequality and condition 3 of the theorem, $1 \succeq \left| S_n(B)(1-A)\right|_{A_p} \succeq \left\| S_n(B)(1-A)\right\|_{L_q} $ for any $ n$, whence $$\frac{1}{l^2} \int_{(a, b) \in \mathcal{P}_l} \max_n \left\|S_n(1-A)\right\|_{L_q}^q \preceq \frac{1}{l^2} \int_{(a, b) \in \mathcal{P}_l} \preceq 1,$$ as required.
\end{proof}
		
\begin{proposition}\label{doubling}
		Let $a, b, c > 0$, and $2a < c$. Then $FB_{a, b, c}(z) \preceq FB_{a, b, \frac{c}{2}}(z)$.
\end{proposition}
\begin{proof}
The domain of  integration in the definition of $FB_{a, b, c}(z)$ splits into the disjoint sets \begin{align*} \mathcal M_1 = \{ ( c^\prime , d^\prime) : \, a < d^\prime - c^\prime < c/2 ; d^\prime \in [ 0 , b/2 ] \} \\ \mathcal M_2 =  \{ ( c^\prime , d^\prime) : c/2  < d^\prime - c^\prime < c ; d^\prime \in [ 0 , b/2 ] \} . \end{align*} 
The integral over $ \mathcal M_1 $ is $ FB_{a, b, \frac{c}{2}}(z)$ and the integral over $ \mathcal M_2 $ is estimated above by $ FB_{a, b, \frac{c}{2}}(z)$ just because $ \eta ( -	d^\prime , -c_1 , c_1 , d^\prime ) \le \eta ( -	d^\prime , -c_2 , c_2 , d^\prime ) $ whenever $ c/2  < d^\prime - c_1 < c $, $ a < d^\prime - c_2 < c/2 $. Since the denominators in the expectations in the integrands are $ \asymp d $ in both integrals, the result follows. 
\end{proof}

	\begin{proposition}\label{long} Let a segment $I$ and a set $\mathcal{M} \subset I, |\mathcal{M}| \ge \frac{98|I|}{99}$, $(u, v) \in \mathcal{P}{|I|}, v - u < \frac{I}{10}$ be given. Then $\forall y \in I$
		\begin{multline*} |\mathbb{E}(y - u, y - v, y + v, y + u)B| \preceq D(u,v) , \\
		D(u, v) : = \frac{1}{|I|}(
		\int_{t \in \mathcal{M}, \eta(a, b, y - u, y - v) \in \mathcal{SYMS}_t} |\mathbb{E}(a, b, y - u, y - v)B|  + \\
		\int_{t \in \mathcal{M}, \eta(y - u, y - v, a, b) \in \mathcal{SYMS}_t} |\mathbb{E}(y - u, y - v, a, b)B| +\\
		\int_{t \in \mathcal{M}, \eta(y + v, y + u, a, b) \in \mathcal{SYMS}_t} |\mathbb{E}(a, b, y + v, y + u)B| + \int_{t \in \mathcal{M}, \eta(y + v, y + u, a, b) \in \mathcal{SYMS}_t} |\mathbb{E}(y + v, y + u, a, b)B|).
\end{multline*} 
\end{proposition}
		\begin{proof}
			Note that if $a < b < c < d < e < f$, then $\eta(a, b, c, d) + \eta(c, d, e, f) = \eta(a, b, e, f)$. Consider cases depending on the magnitude of $v$.
			\begin{enumerate}
				\item If $v \ge \frac{|I|}{6}$, then for $|t| \le \frac{|I|}{30}$, $t - \frac{u - v}{2} \in [-v, v]$ and $t + \frac{u - v}{2} \in [-v, v]$, respectively.
				$$\eta(y - u, y - v, y + u, y + v) =$$
				$$\eta(y - u, y - v, y + t - \frac{u - v}{2}, y + t + \frac{u - v}{2}) +$$
				$$\eta(y + t - \frac{u - v}{2}, y + t + \frac{u - v}{2}, y + v, y + u),$$
				and thus
				$$|\mathbb{E}(y - u, y - v, y + v, y + u)B| \prec$$
				$$|\mathbb{E}(y - u, y - v, y + t - \frac{u - v}{2}, y + t + \frac{u - v}{2})B| +$$
				$$|\mathbb{E}(y + t - \frac{u - v}{2}, y + t + \frac{u - v}{2}, y + v, y + u)B|.$$
				
				Denote by $\mathcal{IT}$ the set of such $t \in \left[-\frac{|I|}{30}, \frac{|I|}{30}\right]$ that $y + \frac{t - \frac{u + v}{2}}{2} \in \mathcal{M}$ and $y + \frac{t + \frac{u + v}{2}}{2} \in \mathcal{M}$. Note that $|\mathcal{IT}| \ge \frac{|I|}{60} - \frac{|I|}{99} \succeq |I|$. Hence,
				$$\mathbb{E}(y - u, y - v, y + v, y + u)B \preceq$$
				$$\frac{1}{|\mathcal{IT}|} \int{t \in \mathcal{IT}}(|\mathbb{E}(y - u, y - v, y + t - \frac{u - v}{2}, y + t + \frac{u - v}{2})B| +$$
				$$|\mathbb{E}(y + t - \frac{u - v}{2}, y + t + \frac{u - v}{2}, y + v, y + u)B|) \preceq D(u, v),$$
				which was to be shown.
				
				\item If $v \le \frac{|I|}{6}$, then for $|t| \ge \frac{2|I|}{5}$, $t - \frac{u - v}{2} \notin [-v, v]$ and $t + \frac{u - v}{2} \notin [-v, v]$, respectively. Assume $y$ is in the left half of the segment, the other case is symmetric.
				
				Then for $t \ge \frac{2}{5}$,
				$$\eta(y - u, y - v, y + v, y + u)B =$$
				$$\eta(y - u, y - v, y + t - \frac{u - v}{2}, y + t + \frac{u - v}{2}) -$$
				$$\eta(y + v, y + u, y + t - \frac{u - v}{2}, y + t + \frac{u - v}{2}),$$
				and thus
				$$|\mathbb{E}(y - u, y - v + y + v, y + u)B| \preceq$$
				$$|\mathbb{E}(y - u, y - v, y + t - \frac{u - v}{2}, y + t + \frac{u - v}{2})B| +$$
				$$|\mathbb{E}(y + v, y + u, y + t - \frac{u - v}{2}, y + t + \frac{u - v}{2})B|.$$
				
				Denote by $\mathcal{IT}$ the set of such $t \in \left[\frac{2|I|}{5}, \frac{|I|}{2}\right]$ that $y + \frac{t - \frac{u + v}{2}}{2} \in \mathcal{M}$ and $y + \frac{t + \frac{u + v}{2}}{2} \in \mathcal{M}$. Then $|\mathcal{IT}| \ge \frac{|I|}{40} - \frac{|I|}{99} \succeq |I|$ and
				$$|\mathbb{E}(y - u, y - v, y + v, y + u)B|
				\preceq$$
				$$\frac{1}{|\mathcal{IT}|} \int_{t \in \mathcal{IT}}(|\mathbb{E}(y - u, y - v, y + t - \frac{u - v}{2}, y + t + \frac{u - v}{2})B| +$$
				$$|\mathbb{E}(y + v, y + u, y + t - \frac{u - v}{2}, y + t + \frac{u - v}{2})B|) \preceq D(u, v),$$
				which was to be shown.
			\end{enumerate}
		\end{proof}
			
	\begin{corollary}\label{measu}
		Let $I$ be a segment and $\mathcal{M} \subset I$ with $|\mathcal{M}| \ge \frac{98|I|}{99}$. Then for all $y \in I$,
		$$
		FB_{\frac{|I|}{2^{400}}, 2^{400}|I|, 2^{400}|I|}(y) \preceq \frac{1}{|I|} \int_{\mathcal{M}} FB_{\frac{|I|}{2^{400}}, 2^{400}|I|, \frac{|I|}{2^{399}}}(z)  dz.
		$$
		\begin{proof}
Substituting the formula for $D$ from  Proposition \ref{long} and changing the orders of integration in each term, we obtain the required result.
				\end{proof}
			\end{corollary}
			
			\begin{lemma}
				If the segment $I$ is not dense, then $\mathcal{SSI}(I) \preceq \mathcal{SAI}(I)$.
				\begin{proof}
					Consider the set $\mathcal{M}$ of points $p$ such that $p \in I, |A(p)| < \frac{99}{100}$. $\frac{|I|}{100} \ge A_0(I) \ge \frac{99}{100}(1 - |\mathcal{M}|)$, hence $|\mathcal{M}| \ge \frac{98|I|}{99}$.
					
					$\mathcal{SSI}(I) = \int_I FB_{\frac{|I|}{2^{400}}, 2^{400}|I|, 2^{400}|I|}(z) \preceq \int_I FB_{\frac{|I|}{2^{400}}, 2^{400}|I|, 2^{399}|I|}(z) \preceq \ldots \preceq \int_I FB_{\frac{|I|}{2^{400}}, 2^{400}|I|, \frac{|I|}{2^{39}}}(z)$ by Proposition \ref{doubling}. But by Corollary \ref{measu}
					$$
					FB_{\frac{|I|}{2^{400}}, 2^{400}|I|, \frac{|I|}{2^{399}}}(z) \preceq \frac{1}{|I|}\int_{\mathcal{M}} |FB_{\frac{|I|}{2^{400}}, 2^{400}|I|, \frac{|I|}{2^{399}}}(z)| , dz,
					$$
					hence
					$$
					\mathcal{SSI}(I) \preceq \frac{1}{|I|} \int_{\mathcal{M}} |FB_{\frac{|I|}{2^{400}}, 2^{400}|I|, \frac{|I|}{2^{399}}}(z)| , dz \preceq \mathcal{SAI}(I),
					$$
					which was to be shown.
				\end{proof}
			\end{lemma}
			
			We call a \bfseries signature \mdseries a quadruple $(len, w, l, r)$.
			
			\begin{lemma}
				Fix an arbitrary $\epsilon > 0$. Then for arbitrary real numbers $l, r, L, R \ge 0$ such that $lL = rR$, there exists a finite set of signatures $\mathcal{S}i = (len_i, w_i, ls_i, rs_i)$ such that
				\begin{enumerate}
					\item $0 \le len_i, ls_i, rs_i \le \max(L, R)$
					\item $\sum |w_i| \le l + r$
					\item $\forall t$, $|l\chi{[0,L]} - \sum w_i \chi_{[ls_i, ls_i + len_i]}|(t) < \epsilon$
					\item $\forall t$, $|r\chi_{[0,R]} - \sum w_i \chi_{[rs_i, rs_i + len_i]}|(t) < \epsilon$.
				\end{enumerate}
				\begin{proof}
					Consider the set $\mathcal{M}$ of all possible quadruples $T = (l, r, L, R)$ such that $lL = rR$. We call $T \in \mathcal{M}$ good if it satisfies the conditions of the lemma. In particular, all quadruples with $r, l \ge \epsilon$ are good, as well as all quadruples with $l = 0$ or $r = 0$ are good, since for such quadruples we can take the empty set of segments. We also define an operator $\mathcal{H}: \mathcal{M} \to \mathcal{M}$ acting on the quadruple $T = (l, r, L, R)$ as follows:
					\begin{enumerate}
						\item if $L < R$, then $\mathcal{H}(T) = (r, l, R, L)$
						\item if $L \ge R$, then $\mathcal{H}(T) = (l, r - l, L - R, R)$
					\end{enumerate}
					
					\begin{proposition}
						For any quadruple $T$, there exists a natural number $n$ such that $\mathcal{H}^{(n)}(T)$ is good.
						\begin{proof}
							Denote $\mathcal{H}^{(k)}(T) = (r_k, l_k, R_k, L_k)$. Notice that for some $n_0$, $\min(l_{n_0}, r_{n_0}) < \epsilon$, since otherwise $r_{k+2} + l_{k+2} \le r_k + l_k - \epsilon$, and thus for some $m$, $r_m + l_m < 0$, which is impossible. If for a given $n_0$, $l_{n_0} = 0$ or $r_{n_0} = 0$, then $\mathcal{H}^{(n)}(T)$ is good. If this is not the case, then $r_{n_0+1} < \epsilon$. After that, applying $\mathcal{H}$ will subtract $r_{n_0+1}$ from $l$ until $l > r_{n_0+1}$. There can be only a finite number of such subtractions, so for some $m$, $l_m, r_m < \epsilon$, i.e., $\mathcal{H}^{(m)}(T)$ is good.
						\end{proof}
					\end{proposition}
					
					\begin{proposition}\label{HTT}
						If $\mathcal{H}(T)$ is good, then so is $T$.
						\begin{proof}
							\begin{enumerate}
								\item if $L < R$, then $\mathcal{H}(T) = (r, l, R, L)$. Take a set of signatures $\mathcal{S}_i = (len_i, w_i, r_i, l_i)$ for the quadruple $(r, l, R, L)$ from the lemma's condition and swap $rs_i$ and $ls_i$ in each signature for each $i$. The resulting set of signatures will be suitable for $T$, so $T$ is good.
								\item if $L \ge R$, then $\mathcal{H}(T) = (l, r - l, L - R, R)$. Take the signature $\mathcal{S} = (R, l, L - R, 0)$. Notice that $l\chi{[0,L]} - l\chi_{[L - R, L]} = l\chi_{[0, L - R]}$ and $r\chi_{[0, R]} - l\chi_{[0, R]} = (r - l)\chi_{[0, R]}$, and also the sum $r + l$ in $\mathcal{H}(T)$ is less than the corresponding sum in $T$ by exactly $l$. Thus, if we take a set of signatures for $\mathcal{H}(T) = (l, r - l, L - R, R)$ and add $\mathcal{S}$ to it, the resulting set will be suitable for $T$, so $T$ is good.
							\end{enumerate}
						\end{proof}
					\end{proposition}
					
					From the two propositions above, it follows that all quadruples are good, which completes the proof of Lemma \ref{HTT}.
				\end{proof}

			\begin{corollary}\label{SigSet}
				
				Fix an arbitrary $\epsilon > 0$. Then for arbitrary real $l, r, L, R \ge 0$ such that $lL = rR$, there exists a finite set of signatures $\mathcal{S}_i = (len_i, w_i, ls_i, rs_i)$ such that
				\begin{enumerate}
					\item $0 \le len_i, ls_i, rs_i \le \max(L, R), len_i \ge \frac{\min(L, R)}{3}$
					
					\item $\sum |w_i| \le 2(l + r)$,
					
					\item $\forall t$ $|l\chi{[0, L]} - \sum w_i \chi_{[ls_i, ls_i + len_i]}|(t) < \epsilon$
					
					\item $\forall t$ $|r\chi_{[0, R]} - \sum w_i \chi_{[rs_i, rs_i + len_i]}|(t) < \epsilon$.
				\end{enumerate}
				
				\begin{proof}
					We use Lemma \ref{HTT} and consider the set of signatures $\mathcal{S}'_i$ that it provides. If for some signature in it $len_i < \frac{\min(L, R)}{3}$, we replace it with 2 signatures $\left(\frac{1}{3} + len_i, w_i, l{i,1}, r_{i,1}\right)$ and $\left(\frac{1}{3}, -w_i, l_{i,2}, r_{i,2}\right)$, where
					\begin{enumerate}
						\item If $l_i \ge \frac{1}{3}$, then $l_{i,1} = l_i - \frac{1}{3}$, $l_{i,2} = l_i - \frac{1}{3}$.
						
						\item If $l_i < \frac{1}{3}$, then $l_{i,1} = l_i$, $l_{i,2} = l_i + len_i$.
						
						\item Similarly (replacing $l_i, L$ with $r_i, R$) we define $r_{i,1}, r_{i,2}$.
					\end{enumerate}
					By doing this for all $i$, we obtain a set of signatures satisfying the condition of the corollary.
				\end{proof}
			\end{corollary}
			
			\begin{proposition}\label{sigfinal}
				
				Fix an interval $I$, a real $\beta$, $\frac{1}{2^{100}} \le \beta \le 2^{100}$, and $\epsilon > 0$. Then there exists a finite set of five-number sets ${\alpha_i, ac_i, bc_i, cc_i, dc_i}$ such that:
				\begin{enumerate}
					
					\item $0 \le ac_i, bc_i, cc_i, dc_i \le 2^{105}$
					
					\item $\sum_i |\alpha_i| \preceq |I|$
					
					\item $dc_i - cc_i = bc_i - ac_i \ge \frac{1}{2^{105}}$
					
					\item for any SO $\eta(a, b, c, d) \in OU(I) \cap \mathcal{SR}_\beta$, the estimate holds
					$$
					|\eta(a, b, c, d) - \sum_i \frac{\alpha_i}{b - a} \eta\left(a + (b - a)ac_i, a + (b - a)bc_i, c - (b - a)cc_i, c - (b - a)dc_i\right)| \succeq \epsilon
					$$
				\end{enumerate}
				Recall that
				$$
				OU(I) = {\eta(a, b, c, d): \eta(a, b, c, d) \in \mathcal{PRS}{\frac{|I|}{2^{300}}, \frac{(2^{200} + 1)|I|}{2^{200}}, \frac{|I|}{2^{200}}}, I \subset [b, c]},
				$$
				and $\mathcal{SR}_\beta$ means the set of such SST $\eta(a, b, c, d)$ that $\frac{b - a}{d - c} = \beta$.
				
				\begin{proof}
					Let $l = L = 1, r = \beta, R = \frac{1}{\beta}$ and choose a set of signatures $\mathcal{S}_i = (len_i, w_i, l_i, r_i)$ for $(l, r, L, R)$ by Corollary \ref{SigSet}. Take $\alpha_i = \frac{w_i}{len_i}, ac_i = ls_i, bc_i = ls_i + len_i, cc_i = rs_i, dc_i = rs_i + len_i$ and verify that such $\alpha, ac_i, bc_i, cc_i, dc_i$ are suitable.
					
					Properties 1-3 are checked by straightforward substitution of properties from Corollary \ref{SigSet}. Note that from these properties it also follows that $c - (b - a)cc_i > a + (b - a)bc_i$, i.e., the order in the parameters of the SO in point 4 is correct.
					
					To prove property 4, consider an arbitrary SO $\eta(a, b, c, d) \in OU(I) \cup \mathcal{SR}\beta$. Then
					$$
					\mathcal{D}(a, b, c, d) = \frac{1}{b - a}\chi_{[a, b]} - \frac{1}{d - c}\chi_{[c, d]} =
					$$
					$$
					\frac{1}{b - a}\left(\sum w_i \chi_{[a + (b - a)ls_i, a + (b - a)ls_i + len_i(b - a)]} - \sum w_i \chi_{[c + (a - b)rs_i, c + (a - b)rs_i + (a - b)len_i(d - c)]} + O\left(\frac{\epsilon}{|I|}\right)\right) =
					$$
					$$
					\frac{1}{b - a} \sum_i \alpha_i \mathcal{D}\left(a + (b - a)ac_i, a + (b - a)bc_i, c - (b - a)cc_i, c - (b - a)dc_i\right) + O\left(\frac{\epsilon}{|I|}\right).
					$$
					Hence
					$$
					\eta(a, b, c, d)(t) = \int_a^t \mathcal{D}(a, b, c, d) =
					$$
					$$
					\int_a^t \left(\frac{1}{b - a} \sum_i \alpha_i \mathcal{D}\left(a + (b - a)ac_i, a + (b - a)bc_i, c - (b - a)cc_i, c - (b - a)dc_i\right) + O\left(\frac{\epsilon}{|I|}\right)\right) =
					$$
					$$
					\sum_i \frac{\alpha_i}{b - a} \eta\left(a + (b - a)ac_i, a + (b - a)bc_i, c - (b - a)cc_i, c - (b - a)dc_i\right)(t) + O(\epsilon),
					$$
					which completes the proof.
				\end{proof}
			\end{proposition}
			
		\end{lemma}
	
\begin{lemma}

$|I|FI(I) \preceq \mathcal{SSI}(I)$ for any interval $I$.

Recall that $FI(I) = \frac{1}{|I|^4} \int_{\eta(a, b, c, d) \in OI(I)} |\mathbb{E}(a, b, c, d)B|^q$.
\end{lemma}

\begin{proof}

Fix a small $\epsilon > 0$.

$$
FI(I) = \frac{1}{|I|^4} \int_{\eta(a, b, c, d) \in OI(I)} |\mathbb{E}(a, b, c, d)B|^q \preceq \frac{1}{|I|^4} \int_{\frac{1}{10} \le \beta \le 10} \int_{\eta(a, b, c, d) \in OI(I) \cup \mathcal{SR}\beta} |\mathbb{E}(a, b, c, d)B|^q.
$$

Fix $\frac{1}{2^{100}} \le \beta \le 2^{100}$ and choose a set of sets ${\alpha_i, ac_i, bc_i, cc_i, dc_i}$ by Corollary \ref{sigfinal} for given $\beta$ and $\epsilon$.

Note that
$$
\int_{\eta(a, b, c, d) \in OI(I) \cup \mathcal{SR}_\beta} |\sum_i \frac{w_i}{b - a}\mathbb{E}(a + (b - a)ac_i, a + (b - a)bc_i, c - (b - a)cc_i, c - (b - a)dc_i)|^q \preceq \mathcal{SSI}(I),
$$
since $\eta(a + (b - a)ac_i, a + (b - a)bc_i, c - (b - a)cc_i, c - (b - a)dc_i) \in \mathcal{PRS}_{\frac{|I|}{2^{400}}, |I|2^{400}, |I|2^{400}}$ by the previous proposition, and the Jacobian of the change of variables
$$
(a, b, c) \to (a + (b - a)ac_i, a + (b - a)bc_i, c - (b - a)cc_i)
$$
is equal to $\theta(1)$ (exercise for the reader).

Now, to complete the proof of the lemma, it is sufficient to sum over $i$ and integrate over $\beta$, and then pass to the limit as $\epsilon \to 0$.
\end{proof}
		
	 \section{Completion of the proof of the theorem}
	
	Now we can formulate an analogue of Lemma \ref{Feuer}, necessary to complete the proof. Namely:
	
	\begin{lemma}\label{BSai}
		Let $I = [a, b]$ be an interval that fell into the dyadic decomposition tree at one of the phases, $I \in \mathcal{NL}$. Then $|I| |E_I B|^q \preceq_K \mathcal{SAI}(I)$.
	\end{lemma}
	
	\begin{proof}
		
		\begin{enumerate}
			
			\item Suppose the interval $I$ is not dense. Then by Proposition 24, its endpoints are balanced. We choose intervals $I_a, I_b$ from the definition of the balanced point for $a, b$ such that $\frac{|I|}{2^{200}} \leq |I_a|, |I_b| \leq \frac{|I|}{2^{210}}$. Consider the set $$U =\{ x < y \in I_a, z < t \in I_b, y - x > \frac{I_a}{100}, t - z > \frac{I_b}{100} \}$$ and consider $\mathcal{TI}(I) = \int_U |(T(t) - T(z)) - (T(y) - T(x))|^q$ (recall that $S = \int B, T = \int S$). Then $\mathcal{TI}(I) \succeq |(S(1) - (T(y) - T(x)))|^q \succeq |S(1) - S(0)|^q = |I|^q |E_I B|^q$ by the definition of a balanced point. On the other hand, $(T(t) - T(z)) - (T(y) - T(x)) = \int_{\eta(x, y, z, t)} B \preceq |I| \mathbb{E}(x, y, z, t) B$ and $U \subset OU(I)$, which means $\mathcal{TI}(I) \preceq |I|^q \int_U |\mathbb{E}(x, y, z, t) B|^q \preceq |I|^q \int_{OU(I)} |\mathbb{E}(x, y, z, t) B|^q = |I|^q F_I(I)$. Now by Lemmas 37 and 31, $|I| F_I(I) \preceq \mathcal{SSI}(I) \prec \mathcal{SAI}(I)$, hence $|I| |E_I B|^q \preceq \mathcal{SAI}(I)$.
			
			\item Now let us assume $I$ is dense. The interval has reached the last step of the algorithm, and in particular, this means that $1000KC \geq \mathbb{E}_I B \geq \frac{C}{1000K}$, since $A_0(I) \geq \frac{1}{1000}$, and it was not declared a leaf in previous steps. Consider the interval $I_1 = \left[ \frac{a + b}{2} - \frac{4}{K^2}, \frac{a + b}{2} - \frac{3}{K^2} \right]$. Suppose there exists a point $u$ in it such that $F_B\left( \frac{|I|}{2^{600}}, 2^{600} |I|, 2^{600} |I| \right) < \frac{C^q}{10^{10q} K^{10q}}$. Choose in this case an interval $I'$ from the definition of a balanced point for point $a$ of length between $\left[ \frac{1}{100K^2}, \frac{1}{K^2} \right]$. Reflect it with respect to point $u$, and obtain the interval $I''$. For a point $t$ in the interval $I'$, we will denote the reflected point as $t'$. By Lemma \ref{midpoint}, there exists a function $F$ such that in any interval $I_2$ of length $|I'|$, there exists a balanced point $p$ with $\mathcal{IF}(I_2) = S(p) = \int_U F(a, b) \frac{T(b) - T(a)}{b - a}$, and since $I'$ is chosen as an interval from the definition of a balanced point for $a$, then $\mathcal{IF}(I') = S(0)$. On the other hand,
			$$
			|\mathcal{IF}(I'') - \mathcal{IF}(I')| \preceq \int_U \left| \frac{T(a') - T(b') - (T(b) - T(a))}{b - a} \right| \preceq
			$$
			$$
			\int_U |E(a, b, b', a')| \preceq \left( F_B\left( \frac{|I|}{2^{600}}, 2^{600} |I|, 2^{600} |I| \right) \right)^{\frac{1}{q}} \leq \frac{C}{10^{10} K^{10}},
			$$
			where the penultimate inequality holds by the power means inequality.
			
			Choose now a balanced point $p \in I'', S(p) = \mathcal{IF}(I'')$. Then at this point, $|S(p) - S(a)| \preceq \frac{C}{K^{10}}$, that is, $|E_{[p, b]} B| = |S(b) - S(p)| \geq \left(1 - \frac{3}{K}\right) |S(b) - S(a)| = E_I B$, and on the other hand, $|p - b| \preceq \frac{b - a}{K^2}$. Thus, the interval $[p, b]$ satisfies the conditions for a subinterval described in step 5 of the algorithm, that is, the interval $[p, 1]$ will be placed in $\mathcal{AL}$ - in any case, the interval $I$ does not fall under the conditions of the lemma.
			
			\item If such a point $u$ does not exist, then
			$$
			\mathcal{SAI}(I) = \int_I F_B\left(\frac{|I|}{2^{600}}, 2^{600} |I|, 2^{600} |I|\right)(z) |1 - A(z)|^q , dz \succeq
			$$
			$$
			\int_{I_1} F_B\left(\frac{|I|}{2^{600}}, 2^{600} |I|, 2^{600} |I|\right)(z) |1 - A(z)|^q \succeq_K C^q \int_{I_1} |1 - A(z)|^q.
			$$
			But note that $\mathbb{E}I A < \frac{1}{2}$, and hence $\mathbb{E}{I_1} A < \frac{2}{3}$, otherwise the interval would have been declared a leaf at step 4 of the algorithm. Therefore, by the mean inequality,
			$$
			C^q \int_{I_1} |1 - A(z)|^q \succeq_K C^q |I|,
			$$
			and at the same time, $|\mathbb{E}_I B| \preceq_K 1$, which completes the proof of the lemma.
			
		\end{enumerate}
	\end{proof}
	
	Now we are ready to complete the proof of the theorem in the same way as in the particular case.
	
	Consider a successful phase of the algorithm. Fix some $p_1$, $\frac{1 + \sqrt{5}}{2} < p_1 < p$, $q_1 = \frac{p_1}{p_1 - 1}$, and also $\alpha$ such that $2 - p_1 > \alpha > 0$.
	
	Note that since by Lemma \ref{SumBal}, $\sum_{I \in \mathcal{NL}, A_0(I) \geq \frac{|I|}{10^{1000} K^{1000}}} \mathcal{BAL}(I) \succeq_K 1$, then for some $M_1$, it is true that
	$$
	\sum_{I \in \mathcal{NL}, A_0(I) \geq \frac{|I|}{10^{1000} K^{1000}}, \frac{1}{M_1} \geq DA(I, \frac{1}{10^{100} K^{100}}) \geq \frac{1}{2M_1}} |I| DA(I, \frac{1}{10^{100} K^{100}})^2 \succeq \frac{1}{C M_1^{\alpha}}.
	$$
	Take this $M_1$ and call the intervals from this set interesting. Denote the set of such intervals by $\mathcal{IS}$.
	
	Also note that if any two interesting intervals intersect, they differ in length by at least a factor of $\frac{3}{2}$. Thus, we can break the set of interesting intervals into 2 such that within each, any intersecting intervals differ in length by at least a factor of 2. For one of these 2 subsets of the partition, the same equality holds; we will narrow the set of interesting intervals to this subset. By $\mathcal{IS}_i$ we will denote interesting intervals of length from $\frac{2\pi}{2^i}$ to $\frac{2\pi}{2^{i+1}}$.
		
	 Let's now define the function $F_i$ as follows:
	
	\begin{enumerate}
		
		\item $F_i(x) = 0, x \notin \mathcal{IS}_i$.
		
		\item Define the function on some $I \in \mathcal{IS}_i$. In the definition of $DA(I, \frac{1}{10^{100}K^{100}})$, the maximum over subintervals is involved. Choose the subinterval $I_1$ where the maximum is achieved, and from all such, the shortest one. Then $F_i(x) = 1, x \in I_1, F_i(x) = c, x \in I \setminus I_1$,
		where the constant $c$ is chosen such that $\mathbb{E}_IF_i = 0$.
		
	\end{enumerate}
	
	It is easy to see that $\mathcal{TS}_i = (F_i, \mathcal{IS}_i)$ is a $T$-system. Let's verify that the condition of Lemma 2 is satisfied for sufficiently large $C$ for given numbers $p, r = q_1, \alpha, M = C^{p_1 - p}M_1, C$, as well as the $T$-sequence $\mathcal{P} = \{\mathcal{TS}_0, \mathcal{TS}_1, \mathcal{TS}_2, \ldots\}$ and function $A$.
	
	\begin{enumerate}
		
		\item First condition holds since $\mathbb{E}(F_i A) = \frac{1}{2\pi} \sum_{I \in \mathcal{IS}_i} |I| \mathbb{E}I(F_i A) \succeq_K \sum{I \in \mathcal{IS}_i} |I| \succeq \frac{\mu_i}{M}$.
		
		\item By Lemmas \ref{Disj1} and \ref{BSai} and Proposition \ref{BeqA}, $$1 \succeq_K \sum_{\mathcal{IS}i} \mathcal{SAI}(I) \succeq_K \sum_{\mathcal{IS}i} |I||E_IB|^q \succeq_K \sum_{\mathcal{IS}i} |I|C^q|E_IA|^q \succeq_K \sum_{\mathcal{IS}i} |I|C^q$$. Hence, $\mu_i \preceq \sum_{\mathcal{IS}_i} |I| \preceq_K C^{-q}$, i.e., $\mu_i \preceq C^{-q_1}$ for sufficiently large $C$, and second condition holds.
		
		\item $\sum_{I \in \mathcal{IS}i} |I|DA(I, \frac{1}{10^{100}K^{100}})^2 \succeq_K \frac{1}{CM_1^{\alpha}}$ by the definition of interesting intervals, hence $\sum_i \mu_i = \sum_{I \in \mathcal{IS}} |I| \succeq_K \frac{M_1^{2 - \alpha}}{C}$, i.e., $\sum_i \mu_i \ge \frac{M^{2 - \alpha}}{C}$ for sufficiently large $C$, so the last condition holds aswell.
		
	\end{enumerate}
	
	Applying Lemma \ref{T-sec} now, we obtain that $\left|A\right|_{A_p} \succeq C^{q_1(2 - p) - 1}$ for any $q_1 > q$, i.e., $ \left| A \right|_{A_p} \succeq C^{q(2 - p) - 1 - \epsilon}$ for any $\epsilon > 0$ for sufficiently large $C$, which completes the proof of Theorem \ref{main}.
	
The condition of Theorem \ref{main} can be somewhat weakened. Here is one of the possible variants.
	
	\begin{proposition}
		
		Let $\delta > 0$ be a fixed positive number, and let $F: \mathbb{R} \to \mathbb{R}$ such that there exists an interval $I$, $|I| < 1 - \delta$ such that $|F(x)| \succeq 1, x \notin I$. Let also the complex-valued $G$ such that $|G(z)| \succ F(\Re z)$. Then condition 3 of the theorem can be replaced by
		\begin{itemize}
			\item $\max_m ||G(\widetilde{A})S_m(\widetilde{B})||{A_p} \preceq 1$.
		\end{itemize}
		In particular, for $G = 1 - z$ and $F = 1 - x$, we obtain the original formulation, and for $G = z, F = x$, we get
		\begin{itemize}
			\item $\max_m ||A S_m(\widetilde{B})||{A_p} \preceq 1$.
		\end{itemize}
		
		The proof remains the same.
		
	\end{proposition}
	
	\section{Acknowledgments}
	The author expresses his enormous gratitude to R. Romanov and A. Tselishchev for fruitful discussions and invaluable assistance in writing the article. The author is the winner of the “Leader” competition conducted by the Foundation for the Advancement of Theoretical Physics and Mathematics “BASIS” and would like to thank its sponsors and jury. 

\end{document}